\documentclass[12pt]{article}%
\usepackage{amscd}
\usepackage{amsmath}
\usepackage{amssymb}
\usepackage{amsthm}
\usepackage{epsfig}
\usepackage{amscd}
\usepackage{amsfonts}
\usepackage{graphicx}
\usepackage[usenames,dvipsnames]{color}%
\usepackage[colorlinks,backref]{hyperref}

\providecommand{\U}[1]{\protect\rule{.1in}{.1in}}
\makeatletter

\newtheorem{thm}{Theorem}[section]
\newtheorem{lem}[thm]{Lemma}
\newtheorem{cor}[thm]{Corollary}
\newtheorem{prop}[thm]{Proposition}

\newtheorem{defi}[thm]{Definition}

\newtheorem{rem}[thm]{Remark}

\renewcommand{\epsilon}{\varepsilon}

\renewcommand{\i}{\mathsf{i}}

\renewcommand{\Re}{\operatorname{Re}}

\renewcommand{\i}{\mathrm{i}}

\makeatother
\begin{document}

\title{Generalized Post-Widder inversion formula with application to statistics}
\author{Denis Belomestny$^{1}$, Hilmar Mai$^{2}$, John Schoenmakers$^{3}$}
\maketitle

\begin{abstract}
In this work we derive an inversion formula for the Laplace transform of a density observed on
a curve in the complex domain, which generalizes the well known Post-Widder
formula. We establish convergence of our inversion method and derive the
corresponding convergence rates for the case of a Laplace transform of a smooth
density. As an application we consider the problem of statistical inference
for variance-mean mixture models. We construct a nonparametric estimator for the mixing density based on the generalized Post-Widder formula, derive bounds for its root mean square error and give a brief numerical example.
\\[0.2cm]\emph{Keywords:} Laplace transform,
inversion formula, Post-Widder formula, variance-mean mixtures, density estimation.

\end{abstract}

\footnotetext[1]{Duisburg-Essen University, Essen and IITP RAS, Moscow, Russia
\texttt{{denis.belomestny@uni-due.de}}}

\footnotetext[2]{CREST ENSAE ParisTech, Paris, \texttt{hilmar.mai@ensae.fr}, (corresponding author)}

\footnotetext[3]{Weierstrass Institute, Berlin,
\texttt{{schoenma@wias-berlin.de}}}

\section{Introduction}

Let $p$ be a probability density on $\mathbb{R}_{+}$, then the integral%
\begin{equation}
\mathcal{L}(z):=\int_{0}^{\infty}e^{-zx}p(x)\,dx,\text{ \ \ }\operatorname{Re}%
z>0, \label{LT}%
\end{equation}
exists and is called the Laplace transform of $p.$ The Laplace transform is a
popular tool for solving differential equations and convolution integral
equations. Its inversion is of importance in many problems from e.g.  physics, engineering and finance (c.f. \cite{bellman1966} and \cite{lavrentev1986} for various examples). 

In general, the complexity of the inversion problem for
$\mathcal{L}$ depends on the information available about the Laplace transform.
If the Laplace transform is explicitly given on its half-plane of convergence,
the density $p$ can be reconstructed using the so-called Bromvich contour
integral (see, e.g. \cite{widder1946})
\[
p(x)=\frac{1}{2\pi{\i}}\int_{c-{\i}\infty}^{c+{\i}\infty}e^{zx}%
\mathcal{L}(z)\,dz,\text{ \ \ }x>0.
\]
In the real case, i.e. in the situation where the Laplace transform of $p$ is known on the real axis only, the inversion of $\mathcal{L}$ is a well-known ill-posed problem (see for
example \cite{kryzhniy2003}, \cite{kryzhniy2006} and references therein). One popular solution for
this case is given by the well known Post-Widder formula which reads as
follows (cf. \cite{widder1946}):
\[
p(x)=\lim_{N\rightarrow\infty}\frac{(-1)^{N}}{N!}\left(  \frac{N}{x}\right)
^{N+1}\mathcal{L}^{(N)}\left(  \frac{N}{x}\right)  .
\]
In some situations, the Laplace transform $\mathcal{L}$ can only be computed
on some curve $\ell$ in $\mathbb{C}$, which is different from $\mathbb{R}_{+}$
or $\{\Re(z)=c\}$ for some $c>0.$ In this paper we generalize the Post-Widder
formula to the case of rather general curves $\ell$ and derive the convergence rates
of the resulting estimator. 

As an application of our results we consider the problem of estimating the mixing density in a variance mean mixture model (see e.g. \cite{barndorff1982normal} and \cite{bingham2002semi}). After constructing the estimator we derive bounds for its root mean square error (RMSE) and demonstrate its performance in a short numerical example. An advantage of using the generalized Post-Widder formula here is that the resulting estimator can be evaluated without any numerical integration.

The paper is organized as follows. In Section \ref{sec:pw} we introduce the generalized Post-Widder inversion formula and discuss its convergence behavior. Section \ref{sec:Estimator-and-convergence} is devoted to the statistical inference for variance mean mixtures together with some numerical results. Finally, the proofs of our results are given in Section \ref{sec:proofs} to \ref{sec:rates}.

\section{Generalized Post-Widder Laplace inversion}

\label{sec:pw}

In this section we will introduce a generalized Post-Widder inversion formula
that extends the classical result by Post and Widder \cite{widder1946} to the
situation when the Laplace transform of a continuous density on $[0,\infty]$
is given on a curve in the complex plane. Subsequently, we prove a convergence
result and derive the rates of convergence for the resulting inverse Laplace transform.

\subsection{Inversion formula and its kernel representation}
Let $p$ be a continuous probability density on $[0,\infty)$ and let its
Laplace transform $\mathcal{L}(z)$ be given on a curve:
\begin{equation}
\ell:=\left\{  z=y+{\i}c(y): \, y\in\mathbb{R}_{+}\right\}  , \label{cu}%
\end{equation}
such that $c$ is piecewise smooth with $c(y)=o(y)$ as $y\rightarrow\infty$. In
this setting the generalized Post-Widder formula can be described as follows.

\begin{defi}
[Generalized Post-Widder formula]For any fixed $x>0,$ we introduce the
generalized Post-Widder formula by
\begin{equation}
p_{N}(x):=\frac{(-1)^{N}}{N!}\left(  g\left(  \frac{N}{x}\right)  \right)
^{N+1}\mathcal{L}^{(N)}\left(  g\left(  \frac{N}{x}\right)  \right)  ,
\label{pN}%
\end{equation}
where $\mathcal{L}^{(N)}$ denotes the $N$th-derivative of the Laplace
transform $\mathcal{L}$ and $g(y):=y+\i c(y).$ For fixed $x>0,$ we define the
generalized Post-Widder kernel via
\begin{equation}
K_{N}(t,x):=\frac{\left(  N+{\i}x c\left(  \frac{N}{x}\right)  \right)
^{N+1}}{N!}t^{N}e^{-\left(  N+{\i}c\left(  \frac{N}{x}\right)  x\right)
t},\text{ \ \ }t>0. \label{genPW}%
\end{equation}

\end{defi}

Our first result deals with the convergence of $p_{N}$ to $p$ as $N \to\infty
$. Such a convergence follows from the properties of the generalized
Post-Widder kernel $K_{N}$ and a representation formula for $p_{N}$ in terms
of $K_{N}$ and $p$. The latter representation is given by the following proposition.

\begin{prop}
\label{pNK} It holds
\[
p_{N}(x)=\int_{0}^{\infty}p(tx)K_{N}(t,x)\,dt.
\]

\end{prop}

The following result states that $K_{N}(t,x)$ converges to the delta function
$\delta(t-1)$ on $(0,\infty)$ for any fixed $x>0.$

\begin{prop}
\label{PWK} The following statements hold.

\begin{itemize}
\item[(i)] For $r=0,1,2,$ it holds%
\begin{align}
\int_{0}^{\infty}t^{r}K_{N}(t,x)dt  &  =\frac{\left(  1+r/N\right)  \cdot
\cdot\cdot(1+1/N)}{\left(  1+{\i}(x/N)c(N/x)\right)  ^{r}}\label{mr}\\
&  =1+\frac{r}{N}-{\i}\frac{r\,c(N/x)}{N/x}+O\left(  \frac{r}{N}%
+\frac{r\,c(N/x)}{N/x}\right)  ^{2},\text{ \ \ }N\rightarrow\infty,\text{
\ \ }x>0. \label{mr1}%
\end{align}
Hence, in particular we have%
\begin{equation}
\int_{0}^{\infty}K_{N}(t,x)dt=1\text{ \ for all }N\in \mathbb{N}. \label{norm}%
\end{equation}

\item[(ii)] Let $x>0$ be fixed. For any $\delta\in(0,1),$ there exists a
natural number $N_{\delta}^{x}$ such that
\[
\int_{\{\left\vert t-1\right\vert \geq\delta,\, t\geq0\}}t^{r}\left\vert
K_{N}(t,x)\right\vert dt\leq Ce^{N\left(  \ln\left(  1+\delta\right)
-\delta\right)  /8},\quad N>N_{\delta}^{x}.
\]
for $r=0,1,2,$ and some constant $C$ not depending on $x$ and $\delta.$
\end{itemize}
\end{prop}

\subsection{Convergence analysis}

By combining Proposition~\ref{pNK} and Proposition~\ref{PWK}, the point-wise
convergence of $p_{N}$ to $p$ follows for $N \to\infty$ as stated in the
following corollary.

\begin{cor}
For any fixed $x\geq0$ and any continuous density $p$ on $[0,\infty),$ we
have
\begin{equation}
\lim_{N\rightarrow\infty}p_{N}(x)=p(x). \label{lim}%
\end{equation}

\end{cor}

We may now sharpen the statement (\ref{lim}) under additional smoothness
assumptions on the density $p.$ In the following propositions we give explicit
convergence rates for $p_{N}$ as $N\rightarrow\infty$. It turns out that the
rates crucially depend on the growth behavior of the function $c(y)$ as
$y\rightarrow\infty$. We henceforth assume that
\begin{equation}
\gamma:=\underset{y\rightarrow\infty}{\lim\sup}\left[  \frac{c^{2}(y)}%
{y}\right]  <\infty.\label{gamma}%
\end{equation}
The notation $f(x,N)=O_{x}(r(x,N))$ for fixed $x\in\mathbb{R}$ and
$N\rightarrow\infty$ means in the sequel the usual $O$-notation where the
actual order coefficient may depend on $x.$ We start with a local Lipschitz
condition on $p$.

\begin{prop}
\label{prop:rates1} \label{propc} Let $p$ be a locally Lipschitz continuous
density on $[0,\infty)$ with Laplace transform (\ref{LT}) given on the curve
(\ref{cu}). It then holds\
\[
p_{N}(x)=p(x)+\mathcal{R}_{N}(x),
\]
where
\[
\mathcal{R}_{N}(x)=O_x(N^{-1/2})
\]
for $N\rightarrow\infty$ and each $x>0$.
\end{prop}

When the density $p$ is differentiable, the rates of Proposition
\ref{prop:rates1} can be improved as the following result shows.

\begin{prop}
\label{propc1} Let $p$ be a differentiable density on $[0,\infty)$ with
Laplace transform (\ref{LT}) given on the curve (\ref{cu}). We then have for
\thinspace$0\leq\gamma<\infty$,%
\[
\operatorname{Re}\left[  p_{N}(x)\right]  =p(x)+\mathcal{R}_{N}(x),
\]
where
\[
\mathcal{R}_{N}(x)=o(N^{-1/2})
\]
for $N\rightarrow\infty$ and each $x>0$. Further we have that%
\begin{align*}
p_{N}(x)  & =p(x)+\mathcal{O}_{x}(N^{-1/2})\text{ \ \ for \ \thinspace
}0<\gamma<\infty,\text{ \ \ and}\\
p_{N}(x)  & =p(x)+o(N^{-1/2})\text{ \ \ for \ \ }\gamma=0.
\end{align*}

\end{prop}

We conclude this section by considering the Laplace inversion problem for a
differentiable density $p$ with locally Lipschitz derivative. It turns out
that we can achieve the error term \(\mathcal{R}_{N}\) of the order $N^{-1}$ in this case.

\begin{prop}
\label{propc2} Let $p$ be a smooth density on $[0,\infty)$ such that its
derivative $p^{\prime}$ is locally Lipschitz, and its Laplace transform
(\ref{LT}) is given on the curve (\ref{cu}). Then for \thinspace$0\leq
\gamma<\infty$,
\[
\operatorname{Re}\left[  p_{N}(x)\right]  =p(x)+\mathcal{R}_{N}(x),
\]
where
\[
\mathcal{R}_{N}(x)=O(N^{-1})
\]
for $N\rightarrow\infty$ and each $x>0$. Moreover we have that%
\begin{align*}
p_{N}(x)  & =p(x)+\mathcal{O}_{x}(N^{-1/2})\text{ \ \ for \ \thinspace
}0<\gamma<\infty,\text{ \ \ and}\\
p_{N}(x)  & =p(x)+O_{x}(N^{-1})\text{ \ \ for \ \ }\gamma=0.
\end{align*}

\end{prop}

In the next section we discuss some applications of the generalized
Post-Widder formula \eqref{pN}.

\section{Application to statistical inference for variance-mean mixtures} \label{sec:Estimator-and-convergence}

The problem of inverting a Laplace transform that is given on a curve $\mathcal{\ell}$ in the complex domain appears naturally in the context of statistical inference for variance-mean mixture models. In this section we apply our generalized Post-Widder Laplace inversion formula to estimate the mixture density in a variance mean mixture model. 

We start the construction of the estimator from the empirical characteristic function that can be written as the Laplace transform of the mixture density evaluated on a certain curve in the complex plain. By inverting this Laplace transform we obtain a nonparametric estimator for the mixing density $p$. Then we derive bounds for the RMSE and conclude by a numerical example.

\subsection{Variance-mean mixture models}

A normal variance-mean mixture model is defined as
\[
q(x):=\int_{0}^{\infty}\frac{1}{\sigma\sqrt{s}}\,\nu\left(  \frac{x-s\mu
}{\sigma\sqrt{s}}\right)  \,p(s)\,ds,
\]
where $\mu\in\mathbb{R},$ $\sigma\in\mathbb{R}_{+}$, $\nu$ is the density of a
standard normal distribution and $p$ is a mixing density on $\mathbb{R}_{+}.$
Variance-mean mixture models play an important role in both theory and
practice of statistics. In particular, such mixtures appear as limit
distributions in asymptotic theory for dependent random variables and they are
useful for modeling data stemming from heavy-tailed and skewed distributions,
see, e.g. \cite{barndorff1982normal} and \cite{bingham2002semi}. 

As can be
easily seen, the variance-mean mixture distribution $q$ coincides with the
distribution of the random variable $\sigma\,\sqrt{\xi}\,X+\mu\,\xi,$ where
$X$ is standard normal and $\xi$ is a nonnegative random variable with the
density $p,$ which is independent of $X.$ The class of variance-mean mixture
models is rather large. For example, the class of the normal variance mixture
distributions ($\mu=0$) can be described as follows: $q$ is the density of a
normal variance mixture if and only if $\mathcal{F}[q](\sqrt{u})$ is a
completely monotone function in $u.$ 

\subsection{Estimating the mixing density}

Here we consider the problem of
statistical inference for the mixing density $p$ based on a sample
$X_{1},\ldots,X_{n}$ from the distribution $q.$ The Fourier transform of the
density $q$ is given by
\begin{equation}
\Phi(u):=\mathcal{F}[q](u)=\int_{0}^{\infty}e^{-s\psi(u)}p(s)\,ds=\mathcal{L}%
[p](\psi(u)) \label{Fourier}%
\end{equation}
with $\psi(u):=-{\i}u\mu+u^{2}\sigma^{2}/2$ and from our data we can
directly estimate the Fourier transform of $q,$ e.g. by means of the so-called
empirical Fourier transform:
\begin{equation}
\Phi_{n}(u):=\frac{1}{n}\sum_{k=1}^{n}e^{\i uX_{k}}. \label{emp_phi}%
\end{equation}
Then we end up with the problem of reconstructing the density $p$ from its
empirical Laplace transform observed on the curve
\[
\ell:=\left\{  z=\operatorname{Re}\psi(u)+{\i}\operatorname{Im}%
\psi(u):\,u\in\mathbb{R}_{+}\right\}  ,
\]
where we have $\operatorname{Re}[\psi(u)]=u^{2}\sigma^{2}/2$ and
$\operatorname{Im}[\psi(u)]=-u\mu.$ Note that
\[
\ell=\left\{  z=y+c(y):\,y\in\mathbb{R}_{+}\right\}  \text{ with }%
c(y)=-\mu\sqrt{2y}/\sigma.
\]
If $\sigma\neq0$ than the function $c$ is smooth and satisfies $c(y)=o(y)$ as
$y\rightarrow\infty.$ Moreover it holds
\begin{equation}
\gamma=\underset{y\rightarrow\infty}{\lim\sup}\left[  \frac{c^{2}(y)}%
{y}\right]  =2\mu^{2}/\sigma^{2}.
\end{equation}
Hence if $p$ is a differentiable density on $[0,\infty)$ such that $p^{\prime
}$ is locally Lipschitz, we can apply Proposition~\ref{propc2} to get the following
asymptotic bound
\begin{equation}
p_{N}(x)-p(x)=%
\begin{cases}
O_{x}(1/N), & \mu=0,\\
O_{x}(1/\sqrt{N}), & \mu\neq 0,
\end{cases}
\label{bias}%
\end{equation}
for $p_{N}$ defined in \eqref{pN} with $g(y)=y-{\i}\mu\sqrt{2y}/\sigma.$ Due
to (\ref{Fourier}), we have $\mathcal{L}(z)=\Phi(\xi(z)),$ where $\xi$ is the
inverse of $\psi$ on $\ell.$ Without loss of generality we may assume that
$\sigma=1,$ then $\xi(z)=\sqrt{2z-\mu^{2}}+{\i}\mu$ using the principal
branch of the square root. So, for $l\geq1$ we have that
\[
\xi^{(l)}(z)=\left(  -1\right)  ^{l-1}\frac{\left(  2\left(  l-1\right)
\right)  !}{2^{l-1}\left(  l-1\right)  !}\left(  2z-\mu^{2}\right)  ^{\frac
{1}{2}-l},
\]
and by Faa di Bruno's formula it follows that for $z\in\ell,$
\begin{align}
\mathcal{L}^{(N)}(z)  &  =\sum_{\substack{k_{1},\ldots,k_{N}\geq
0,\\k_{1}+2k_{2}+\ldots+Nk_{N}=N\\k_{1}+k_{2}+\ldots+k_{N}=k}}\frac{N!}%
{k_{1}!\ldots k_{N}!(1!)^{k_{1}}\ldots(N!)^{k_{N}}}\Phi^{(k)}(\xi(z))%
{\displaystyle\prod\limits_{l=1}^{N}}
\left(  \xi^{(l)}(z)\right)  ^{k_{l}}\nonumber\\
&  =\sum_{k=1}^{N}\Phi^{(k)}(\xi(z))\left(  -1\right)  ^{N-k}\left(
2z-\mu^{2}\right)  ^{\frac{1}{2}k-N}F_{N,k}. \label{Faa}%
\end{align}
The coefficients $F_{N,k}$ can be expressed as follows%
\begin{align*}
F_{N,k}  &  :=\sum_{\substack{k_{1},\ldots,k_{N}\geq0,\\k_{1}+2k_{2}%
+\ldots+Nk_{N}=N\\k_{1}+k_{2}+\ldots+k_{N}=k}}\frac{N!}{k_{1}!\ldots k_{N}!}%
{\displaystyle\prod\limits_{l=1}^{N}}
\left(  \frac{\left(  2\left(  l-1\right)  \right)  !}{2^{l-1}\left(
l-1\right)  !l!}\right)  ^{k_{l}}\\
&  =\sum_{\substack{k_{1},\ldots,k_{N}\geq0,\\k_{1}+2k_{2}+\ldots+\left(
N-k+1\right)  k_{N-k+1}=N\\k_{1}+k_{2}+\ldots+k_{N-k+1}=k,}}\frac{N!}%
{k_{1}!\cdot\cdot\cdot k_{N-k+1}!}%
{\displaystyle\prod\limits_{l=1}^{N-k+1}}
\left(  \frac{\frac{\left(  2\left(  l-1\right)  \right)  !}{2^{l-1}\left(
l-1\right)  !}}{l!}\right)  ^{k_{l}}\\
&  =B_{N,k}\left(  1,...,\frac{\left(  2\left(  N-k\right)  \right)
!}{2^{N-k}\left(  N-k\right)  !}\right),
\end{align*}
where $B_{N,k}$ stand for the partial Bell polynomials. In view of
\eqref{emp_phi} and (\ref{Faa}), we now introduce%
\[
\mathcal{L}_{n}^{(N)}(z):=\sum_{k=1}^{N}\Phi_{n}^{(k)}(\xi(z))\left(
-1\right)  ^{N-k}\left(  2z-\mu^{2}\right)  ^{\frac{1}{2}k-N}F_{N,k}%
\]
as an unbiased estimator for $\mathcal{L}^{(N)}(z)$ at every $z\in\ell.$
We so arrive at an empirical estimate for the mixing density $p$: 
\begin{align}
\nonumber
p_{n,N}(x)  &  :=\frac{(-1)^{N}}{N!}\left(  g(N/x)\right)  ^{N+1}%
\mathcal{L}_{n}^{(N)}\left(  g(N/x)\right) \\
\nonumber
&  =\frac{(-1)^{N}}{N!}\left(  g(N/x)\right)  ^{N+1}\frac{1}{n}\sum_{j=1}%
^{n}e^{{\i}\xi(g(N/x))X_{j}}\\
\label{pnN}
&  \times\sum_{k=1}^{N}({\i}X_{j})^{k}\left(  -1\right)  ^{N-k}\left(
2g(N/x)-\mu^{2}\right)  ^{\frac{1}{2}k-N}F_{N,k},
\end{align}
which obviously satisfies $\mathbb{E}\left[  p_{n,N}(x)\right]  =p_{N}(x).$ The
coefficients $F_{N,k}$ can be computed by evaluating the partial Bell
polynomials $B_{N,k}$ that are available in most computational algebra
packages. Hence, we obtain an explicit estimator for $p$ that circumvents the
use of numerical integration procedures as needed in other Laplace inversion techniques.

\subsection{Convergence of the estimator}
Let us now analyze the variance of $p_{n,N}.$ 
\begin{thm}
\label{thm:var}
For some constant $C>1,$ depending on $x>0,$ it holds that
\begin{equation}
\operatorname{Var}\left[  p_{n,N}(x)\right]  \lesssim\frac{C^{N}}{n}\sum_{k=1}%
^{N}N^{-k}\beta_{2k}, \label{var}%
\end{equation}
where $\beta_{2k}:=\mathbb{E}\left[  \left\vert X_{1}\right\vert ^{2k}\right]
$.
\end{thm}

Based on the estimate \eqref{var}, we can derive upper bounds of the root mean
square error (RMSE) for the density estimator $p_{n,N}$.

\begin{thm}
\label{rates} Fix some $x>0$ and suppose that $\mathcal{R}_{N}(x)=p_{N}%
(x)-p(x)=O_{x}(N^{-\rho}).$ We then have the following bounds for the RMSE of
$p_{n,N}.$

\begin{description}
\item[(i)] If $\beta_{2k}\leq A^{k}k^{k}$ for some $A>0$ and all natural
$k>1,$ then
\[
\operatorname{RMSE}(p_{n,N})=O_{x}\left(  \frac{1}{\ln^{\rho}n}\right)  ,
\]
provided
\begin{equation}
N=\frac{\ln n}{\ln(AC)}-\frac{2\rho}{\ln(AC)}\ln\ln n, \label{NC}%
\end{equation}
where the constant $C$ comes from the bound \eqref{var}.

\item[(ii)] In the case $\beta_{2k}\leq A^{k}k^{bk},$ $k\in\mathbb{N}$ for
some $b>1,$ it holds that
\[
\operatorname{RMSE}(p_{n,N})\lesssim\frac{\ln^{\rho}\ln n}{\ln^{\rho}n},
\]
which is achieved by choosing%
\begin{equation}
N=\frac{\left(  2\rho+\ln n\right)  }{\left(  b-1\right)  \ln\left(  \left(
2\rho+\ln n\right)  /\left(  b-1\right)  \right)  }-\frac{2\rho}{b-1}.
\label{NTH}%
\end{equation}

\end{description}
\end{thm}
\begin{rem}
Because of the inequality \(\operatorname{Var}\left[  p_{n,N}(x)\right]\geq \operatorname{Var}\left[  \Re[p_{n,N}(x)]\right],\) the results of Theorem~\ref{rates} remain valid with \(p_{n,N}\) replaced by \(\Re[p_{n,N}].\)
\end{rem}

\subsection{Numerical example}

Let the mixing density $p$ be the exponential density, i.e. $p(x)=\exp(-x),$ then
\begin{align*}
\mathbb{E}\left[  \left\vert X_{1}\right\vert ^{2k}\right]  =\frac{2^{k}%
}{\sqrt{\pi}}\Gamma(1+k)\Gamma(k+1/2) \leq A^{k} k^{2k}%
\end{align*}
for some $A>0.$ Combining Proposition~\ref{propc2} with Theorem~\ref{rates},
we obtain
\[
\operatorname{RMSE}(\Re[p_{n,N}])\lesssim\frac{\ln\ln (n)}{\ln(n)}, \quad
n\to\infty.
\]
\begin{figure}[!tp]
\centering
\includegraphics[width=0.7\linewidth]{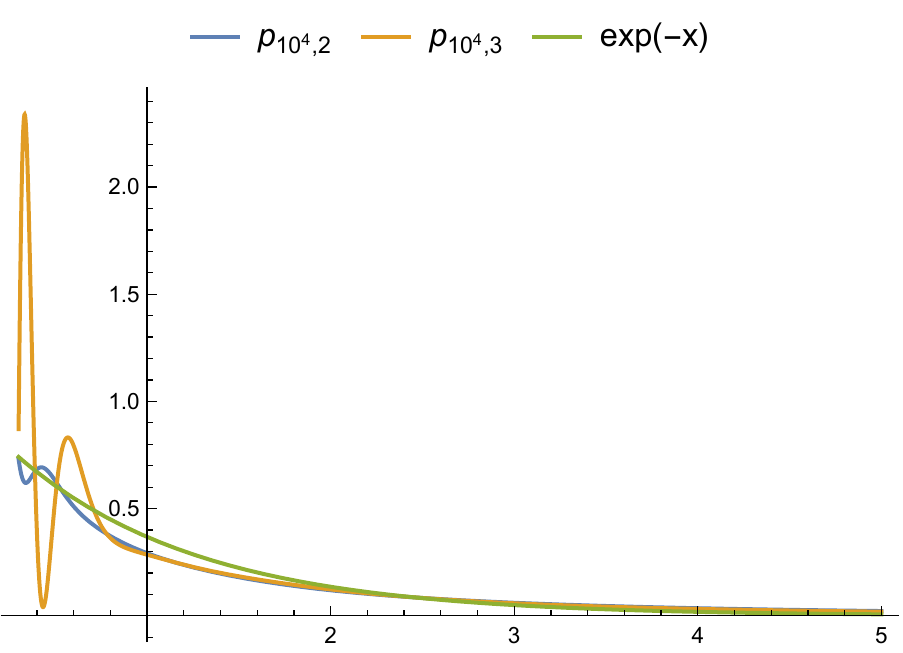}
~\includegraphics[width=0.7\linewidth]{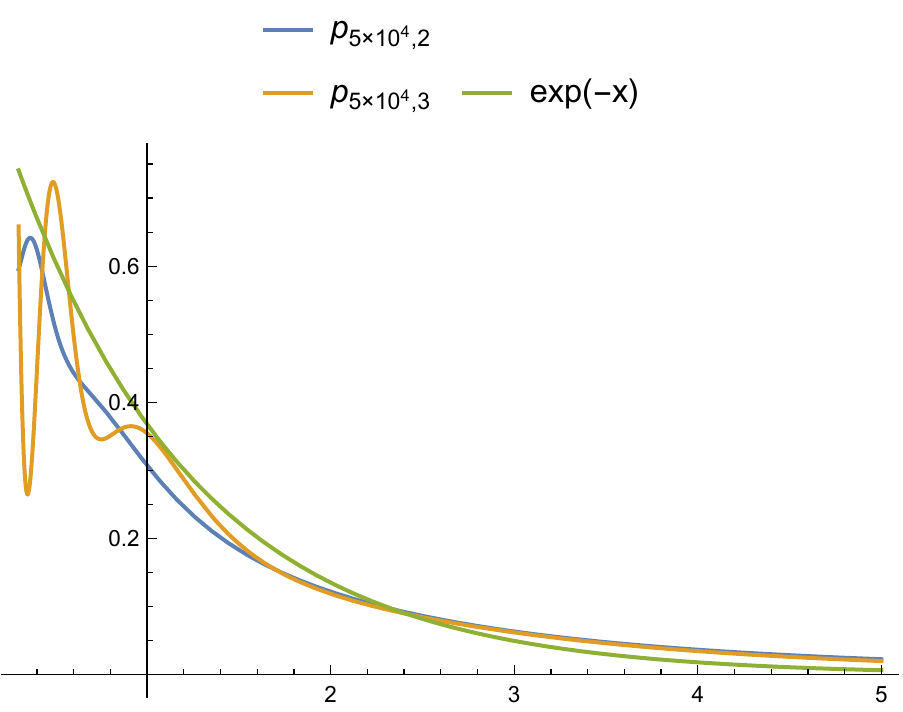}
 \caption{Approximations \(p_{n,N}\) (real parts) for sample sizes \(n=10000\) (above), \(n=50000\) (below)   and different values of \(N\) of the true exponential density \(p(x)=\exp(-x)\) in the normal mean variance mixture model with parameters \(\mu=0.1\) and \(\sigma=1.\)
\label{fig: pnN_estimate}}
\end{figure}
In Figure~\ref{fig: pnN_estimate} one can see the result of numerical estimation of the underlying exponential density \(p(x)=\exp(-x)\) based on different numbers of terms \(N\) in 
\eqref{pnN} and different sample sizes \(n\). As can be observed, the estimation error increases as \(x\to +0.\) This effect can be explained by noting that \(|g(N/x)|\to \infty\) as \(x\to +0\) and so the variance increases for small \(x\) (see the proof of Theorem~\ref{thm:var}).

\section{Acknowledgment}
The first author was supported by the Deutsche
Forschungsgemeinschaft through the SFB 823 ``Statistical modelling of nonlinear dynamic processes'' and by a Russian Scientific Foundation grant (project N 14-50-00150). The second author received founding by the LABEX Louis Bachelier ``Finance et Croissance Durable'' in Paris.

\section{Proofs}

\label{sec:proofs}

In this section we gather the proofs of our results from Section \ref{sec:pw}.

\subsection{Proof of Proposition~\ref{pNK}}

We start by proving the integral representation of $p_{N}$ in terms of $p$ and
the generalized Post-Widder kernel $K_{N}$. We have by definition%
\[
\mathcal{L}(z)=\int_{0}^{\infty}e^{-uz}p(u)du,\quad\operatorname{Re}z>0
\]
differentiating $N$-times results in
\[
\mathcal{L}^{(N)}(z)=\int_{0}^{\infty}(-u)^{N}e^{-uz}p(u)du,
\]
yielding finally
\begin{align*}
p_{N}(x)  &  =\frac{(-1)^{N}}{N!}\left(  \frac{N}{x}+{\i}c\left(  \frac{N}%
{x}\right)  \right)  ^{N+1}\int_{0}^{\infty}(-u)^{N}e^{-u\left(  \frac{N}%
{x}+{\i}c\left(  \frac{N}{x}\right)  \right)  }\, p(u)\, du\\
&  =\frac{1}{N!}\left(  \frac{N}{x}+{\i}c\left(  \frac{N}{x}\right)  \right)
^{N+1}\int_{0}^{\infty}(xt)^{N}e^{-t\left(  N+{\i}xc\left(  \frac{N}%
{x}\right)  \right)  }p(xt)x\, dt\\
&  =\int_{0}^{\infty}\frac{1}{N!}\left(  N+{\i}xc\left(  \frac{N}{x}\right)
\right)  ^{N+1}t^{N}e^{-t\left(  N+{\i}c\left(  \frac{N}{x}\right)  x\right)
}p(xt)\, dt\\
&  =\int_{0}^{\infty}p(tx)K_{N}(t,x)\, dt.
\end{align*}

\subsection{Proof of Proposition~\ref{PWK}}

(i): For $r=0,1,2,...$ we have
\begin{align*}
\int_{0}^{\infty}t^{r}K_{N}(t,x)dt  &  =\int_{0}^{\infty}\frac{1}{N!}\left(
N+{\i}xc\left(  \frac{N}{x}\right)  \right)  ^{N+1}t^{N+r}e^{-t\left(  N+{\i
}xc\left(  \frac{N}{x}\right)  \right)  }dt\\
&  =\frac{1}{N!}\left(  N+{\i}xc\left(  \frac{N}{x}\right)  \right)  ^{-r}%
\int_{0}^{\left(  N+{\i}xc\left(  \frac{N}{x}\right)  \right)  \cdot\infty
}z^{N+r}e^{-z}dz.
\end{align*}
Note that on the set $\bigl\{z:\,\mathrm{arc}(z)=R\,e^{\i\theta},$
$-\pi/2<\theta<\pi/2\bigr\}$ it holds that $\left\vert z^{N+r}e^{-z}%
\right\vert =R^{N+r}e^{-R\,\cos\theta}\rightarrow0$ for $R\rightarrow\infty.$
Thus, by the Cauchy integral theorem,
\[
\int_{0}^{\left(  N+{\i}xc\left(  \frac{N}{x}\right)  \right)  \cdot\infty
}z^{N+r}e^{-z}dz=\int_{0}^{\infty}t^{N+r}e^{-t}dt=\left(  N+r\right)  !
\]
from which (\ref{mr}) follows. Thus (\ref{mr}) holds for any integer integer
$r\geq0.$ The asymptotic expression (\ref{mr1}) for $r=1$ and $r=2$ can be
seen from taking the logarithm of (\ref{mr}):
\begin{align*}
\ln\frac{\left(  1+r/N\right)  \cdot\cdot\cdot(1+1/N)}{\left(  1+{\i
}(x/N)c(N/x)\right)  ^{r}}  &  =\sum_{l=1}^{r}\ln\left(  1+l/N\right)
-r\ln\left(  1+{\i}(x/N)c(N/x)\right) \\
&  =\frac{r}{N}-{\i}\frac{rc(N/x)}{N/x}+O\left(  \frac{1}{N}+\frac
{c(N/x)}{N/x}\right)  ^{2}.
\end{align*}

(ii): By Stirling's formula it follows that%
\begin{align*}
K_{N}(t,x)  &  =\frac{\left(  1+{\i}\left(  \frac{x}{N}\right)  c\left(
\frac{N}{x}\right)  \right)  ^{N+1}N^{N+1}}{\sqrt{2\pi N}\cdot\frac{N^{N}%
}{e^{N}}}t^{N}e^{-Nt}e^{-{\i}(xt)c(\frac{N}{x})}(1+O(1/N))\\
&  =\left(  1+{\i}c\left(  \frac{N}{x}\right)  \frac{x}{N}\right)  ^{N+1}%
\sqrt{\frac{N}{2\pi}}t^{N}e^{-Nt}e^{N}e^{-{\i}(xt) c(\frac{N}{x})}(1+O(1/N))
\end{align*}
and hence%
\begin{align}
\ln K_{N}(t,x)  &  =-{\i}(xt) c\left(  \frac{N}{x}\right)  +(N+1)\ln\left(
1+{\i}\cdot\frac{x}{N}\cdot c\left(  \frac{N}{x}\right)  \right) \nonumber\\
&  +\frac{1}{2}\ln\frac{N}{2\pi}+N\ln t-Nt+N+O(1/N)\nonumber\\
&  =(N+1)\sum_{j=1}^{\infty}\frac{(-1)^{j-1}}{j}\left(  {\i}c\left(  \frac
{N}{x}\right)  \frac{x}{N}\right)  ^{j}+\frac{1}{2}\ln\frac{N}{2\pi
}\nonumber\\
&  +N\ln t-Nt+N-{\i}(xt)c\left(  \frac{N}{x}\right)  +O(1/N)\nonumber\\
&  =\frac{1}{2}\ln\frac{N}{2\pi}+ O\left(  Nc\left(  \frac{N}{x}\right)
\frac{x}{N}\right)  ^{2}+O(1/N)\nonumber\\
&  +N\ln t-Nt+N+{\i}c\left(  \frac{N}{x}\right)  x\left(  1-t\right)  +{\i
}c\left(  \frac{N}{x}\right)  . \label{lnK}%
\end{align}
In particular, for $t\neq1$ we have
\begin{align}
\left\vert K_{N}(t,x)\right\vert  &  =\exp\left[  \frac{1}{2}\ln\frac{N}{2\pi
}+O(1/N)+N\left(  \underset{<0}{\underbrace{\ln t-t+1}}+\underset
{\rightarrow0}{\underbrace{O\left(  c\left(  \frac{N}{x}\right)  \frac{x}%
{N}\right)  ^{2}}}\right)  \right] \nonumber\\
&  \rightarrow0\text{ \ \ \ as }N\rightarrow\infty. \label{KAbs}%
\end{align}
Let us fix $x>0$ and $\delta>0$ arbitrarily. W.l.o.g. we may assume that
$\delta<1.$ Because $c(\frac{N}{x})\frac{x}{N}\rightarrow0$ for $N\rightarrow
\infty,$ there exist a number $N_{\delta}^{x}$ such that for any $N>N_{\delta
}^{x}$ and any $t>0$ with $\left\vert t-1\right\vert \geq\delta,$%
\[
\ln t-t+1+O\left(  c\left(  \frac{N}{x}\right)  \frac{x}{N}\right)  ^{2}%
<\frac{1}{2}\left(  \ln t-t+1\right)  ,
\]
i.e.%
\[
\left\vert K_{N}(t,x)\right\vert \leq\exp\left[  \frac{1}{2}\ln\frac{N}{2\pi
}+O(1/N)+\frac{1}{2}N\left(  \ln t-t+1\right)  \right]  ,
\]
and so for $r=0,1,2,$ $N>N_{\delta}^{x},$%
\begin{equation}
\int_{\{\left\vert t-1\right\vert \geq\delta,\text{ }t\geq0\} }t^{r}\left\vert
K_{N}(t,x)\right\vert dt\leq C_{1}\sqrt{\frac{N}{2\pi}}\int_{\{\left\vert
t-1\right\vert \geq\delta,\text{ } t\geq0\}}t^{r}\left(  te^{-t+1}\right)
^{N/2}dt. \label{dpl}%
\end{equation}
Note that for $\left\vert t-1\right\vert \geq\delta$ we have $te^{-t+1}=e^{\ln
t-t+1}\leq e^{\frac{\ln t-t+1}{2}+\frac{\ln\left(  1+\delta\right)  -\delta
}{2}},$ hence for $r=0,1,2,$ $N>N_{\delta}^{x},$%
\begin{align*}
\int_{\{\left\vert t-1\right\vert \geq\delta,\text{ }t\geq0\}}t^{r}\left(
te^{-t+1}\right)  ^{N/2}  &  \leq e^{N\left(  \ln\left(  1+\delta\right)
-\delta\right)  /4}\int_{\{\left\vert t-1\right\vert \geq\delta,\text{ }%
t\geq0\}}t^{r}e^{N\left(  \ln t-t+1\right)  /4}dt\\
&  \leq e^{N\left(  \ln\left(  1+\delta\right)  -\delta\right)  /4}%
\int_{\{\left\vert t-1\right\vert \geq\delta,\text{ }t\geq0\}}t^{r}e^{\left(
\ln t-t+1\right)  /4}dt\\
&  \leq e^{N\left(  \ln\left(  1+\delta\right)  -\delta\right)  /4}\int
_{0}^{\infty}t^{r+1}e^{\left(  -t+1\right)  /4}dt\leq C_{2}e^{N\left(
\ln\left(  1+\delta\right)  -\delta\right)  /4},
\end{align*}
where $\ln\left(  1+\delta\right)  -\delta<0$ and $C_{2}>0.$ It next follows
from (\ref{dpl}) that for some constant $C>0$%
\begin{equation}
\int_{\{\left\vert t-1\right\vert \geq\delta,\text{ }t\geq0\}}t^{r}\left\vert
K_{N}(t,x)\right\vert dt\leq Ce^{N\left(  \ln\left(  1+\delta\right)
-\delta\right)  /8}\rightarrow0 \label{exs}%
\end{equation}
for $N\rightarrow\infty,$ $N>N_{\delta}^{x}.$

\subsection{Proof of Proposition~\ref{propc}}

In order to derive the convergence rates in (\ref{lim}), we proceed with the
following lemma.

\begin{lem}
\label{ass1} For $r=1,2,..$
\[
\int_{0}^{\infty}\left\vert t-1\right\vert ^{r}\left\vert K_{N}%
(t,x)\right\vert dt\leq C6^{\left(  r+1\right)  /2}\,\frac{\Gamma\left(
\left(  r+1\right)  /2\right)  }{N^{r/2}}\exp\left[  O\left(  c^{2}\left(
\frac{N}{x}\right)  \frac{x^{2}}{N}\right)  \right]  .
\]

\end{lem}

\begin{proof}
Let us fix $x>0.$ For $\delta>0$ with $0<\delta<1$ we have by (\ref{KAbs}),
\begin{align}
\int_{1}^{1+\delta}\left\vert t-1\right\vert ^{r}\left\vert K_{N}%
(t,x)\right\vert dt  &  \leq C\sqrt{\frac{N}{2\pi}}\exp\left[  N\times
O\left(  c\left(  \frac{N}{x}\right)  \frac{x}{N}\right)  ^{2}\right]
\nonumber\\
&  \times\int_{1}^{1+\delta}\left(  t-1\right)  ^{r}\exp\left[  N\left(  \ln
t-t+1\right)  \right]  \,dt, \label{hu1}%
\end{align}
where%
\begin{multline*}
\int_{1}^{1+\delta}\left(  t-1\right)  ^{r}\exp\left[  N\left(  \ln
t-t+1\right)  \right]  \,dt\\
=\int_{0}^{\delta}u^{r}\exp\left[  N\left(  \ln\left(  1+u\right)  -u\right)
\right]  \,du\\
=\frac{6^{\left(  r+1\right)  /2}}{2}\frac{\Gamma\left(  \left(  r+1\right)
/2\right)  }{N^{\left(  r+1\right)  /2}}.
\end{multline*}
Hence by (\ref{hu1}) and Theorem \ref{PWK}-(ii)
\[
\int_{1}^{\infty}\left\vert t-1\right\vert ^{r}\left\vert K_{N}%
(t,x)\right\vert dt\leq C_{1}\frac{6^{\left(  r+1\right)  /2}}{2}\frac
{\Gamma\left(  \left(  r+1\right)  /2\right)  }{N^{r/2}}\exp\left[  O\left(
c^{2}\left(  \frac{N}{x}\right)  \frac{x^{2}}{N}\right)  \right]  .
\]
Similarly,%
\begin{align*}
\int_{1-\delta}^{1}\left\vert t-1\right\vert ^{r}\left\vert K_{N}%
(t,x)\right\vert dt  &  =C\sqrt{\frac{N}{2\pi}}\exp\left[  O\left(
c^{2}\left(  \frac{N}{x}\right)  \frac{x^{2}}{N}\right)  \right] \\
&  \times\int_{1-\delta}^{1}\left(  1-t\right)  ^{r}\exp\left[  N\left(  \ln
t-t+1\right)  \right]  \,dt,
\end{align*}
where%
\begin{align*}
&  \int_{1-\delta}^{1}\left(  1-t\right)  ^{r}dt\exp\left[  N\left(  \ln
t-t+1\right)  \right] \\
&  =\int_{0}^{\delta}u^{r}du\exp\left[  N\left(  \ln\left(  1-u\right)
+u\right)  \right]  \leq\int_{0}^{\delta}u^{r}dt\exp\left[  -Nu^{2}/6\right]
.
\end{align*}
By applying Theorem \ref{PWK}-(ii) once again the statement of the lemma is proved.
\end{proof}

Let us now fix $x>0.$ By the Lipschitz assumption on $p$ we thus have for
fixed $0<\delta<1,$%
\begin{align*}
\left\vert p_{N}(x)-p(x)\right\vert  &  \leq\int\left\vert
p(tx)-p(x)\right\vert \left\vert K_{N}(t,x)\right\vert dt\\
&  \leq K_{x}\left\vert x\right\vert \int_{\left\vert t-1\right\vert <\delta
}\left\vert t-1\right\vert \left\vert K_{N}(t,x)\right\vert dt+K_{1}%
e^{N\left(  \ln\left(  1+\delta\right)  -\delta\right)  /8}\\
&  \leq K_{x}^{1}\left\vert x\right\vert N^{-1/2}\exp\left[  O\left(
c^{2}(\frac{N}{x})\frac{x^{2}}{N}\right)  \right] \\
&  \leq K_{x}^{2}\left\vert x\right\vert N^{-1/2}%
\end{align*}
due to assumption (\ref{gamma}).

\subsection{Proof of Proposition~\ref{propc1}}

Let us fix $x>0.$ By differentiability of $p$ we may find for any $\epsilon>0$
a $\delta_{\epsilon}$ with $0<\delta_{\epsilon}<1$ such that%
\[
p(tx)=:p(x)+(t-1)\left(  xp^{\prime}(x)+E_{t}^{x}\right)  ,
\]
with $\left\vert E_{t}^{x}\right\vert <\epsilon$\ for all $t>0$ with
$\left\vert t-1\right\vert <\delta_{\epsilon}.$Due to Proposition \ref{pNK} we
then have%
\begin{gather*}
p_{N}(x)-p(x)=\int_{0}^{\infty}\left(  p(tx)-p(x)\right)  K_{N}(t,x)dt=\int
_{\left\vert t-1\right\vert \geq\delta}\left(  p(tx)-p(x)\right)
K_{N}(t,x)dt\\
+xp^{\prime}(x)\int_{\left\vert t-1\right\vert <\delta}(t-1)K_{N}%
(t,x)dt+\int_{\left\vert t-1\right\vert <\delta}(t-1)E_{t}^{x}K_{N}%
(t,x)dt=:(\ast)_{1}+(\ast)_{2}+(\ast)_{3}.
\end{gather*}
Since $p$ is bounded we have by Theorem (\ref{PWK})-(ii) that
\begin{equation}
(\ast)_{1}\leq D_{1}e^{N\left(  \ln\left(  1+\delta_{\epsilon}\right)
-\delta_{\epsilon}\right)  /8} \label{c1}%
\end{equation}
for some $D_{1}>0$ and $N>N_{\delta_{\epsilon}}^{x}$ (cf. the proof of Theorem
\ref{PWK})-(ii). By Theorem (\ref{PWK}) we have%
\begin{align}
\int_{\left\vert t-1\right\vert <\delta_{\epsilon}}(t-1)K_{N}(t,x)dt  &
=\int_{0}^{\infty}(t-1)K_{N}(t,x)dt-\int_{\left\vert t-1\right\vert \geq
\delta_{\epsilon},~t\geq0}(t-1)K_{N}(t,x)dt\nonumber\\
&  =\frac{1}{N}-{\i}\frac{c(N/x)}{N/x}+O\left(  \frac{1}{N}+\frac
{c(N/x)}{N/x}\right)  ^{2}+f_{N,\delta_{\epsilon}} \label{c2}%
\end{align}
with $\left\vert f_{N,\delta_{\epsilon}}\right\vert \leq D_{2}e^{N\left(
\ln\left(  1+\delta_{\epsilon}\right)  -\delta_{\epsilon}\right)  /8}$ for
$N>N_{\delta_{\epsilon}}^{x}.$ Next, by Lemma \ref{ass1} and assumption
(\ref{gamma}) we have%
\begin{align}
\left\vert (\ast)_{3}\right\vert  &  \leq\epsilon\int_{\left\vert
t-1\right\vert <\delta}\left\vert (t-1\right\vert \left\vert K_{N}%
(t,x)\right\vert dt\leq D_{3}\frac{\epsilon}{N^{1/2}}\exp\left[  O\left(
c^{2}(\frac{N}{x})\frac{x^{2}}{N}\right)  \right] \label{c3}\\
&  \leq D_{4}\frac{\epsilon}{N^{1/2}}\nonumber
\end{align}
From (\ref{c1})--(\ref{c3}) we gather that%
\begin{equation}
p_{N}(x)=p(x)+xp^{\prime}(x)\left(  \frac{1}{N}-{\i}\frac{c(N/x)}%
{N/x}\right)  +O_{x}(\frac{1}{N}+\frac{c(N/x)}{N/x})^{2}+o(N^{-1/2}).
\label{Ree}%
\end{equation}
Now, since
\begin{equation}
\frac{c(N/x)}{N/x}=O\left(  N^{-1/2}\right)  ,\text{ \ if \ }0<\gamma
<\infty,\text{ \ \ and \ \ }\frac{c(N/x)}{N/x}=o(N^{-1/2})\text{ \ if }%
\gamma=0, \label{reee}%
\end{equation}
the statements follow by taking the real part of (\ref{Ree}).

\subsection{Proof of Proposition~\ref{propc2}}

Let us fix $x>0.$ By the Lipschitz assumption on $p^{\prime}$ we may find a
$\delta$ with $0<\delta<1$ such that%
\[
p(tx)=:p(x)+(t-1)xp^{\prime}(x)+\frac{1}{2}(t-1)^{2}R_{t}^{x},
\]
with $\left\vert R_{t}^{x}\right\vert <R^{x}$\ for some constant $R^{x}>0$ and
for all $t>0$ with $\left\vert t-1\right\vert <\delta.$ Due to Proposition
\ref{pNK} we thus have%
\begin{gather}
p_{N}(x)-p(x)=\int_{0}^{\infty}\left(  p(tx)-p(x)\right)  K_{N}(t,x)dt=\int
_{\left\vert t-1\right\vert \geq\delta}\left(  p(tx)-p(x)\right)
K_{N}(t,x)dt\nonumber\\
+\int_{\left\vert t-1\right\vert <\delta}\left(  (t-1)xp^{\prime}(x)+\frac
{1}{2}(t-1)^{2}R_{t}^{x}\right)  K_{N}(t,x)dt=:(\ast)_{1}+(\ast)_{2}%
.\label{int0}%
\end{gather}
Since $p$ is bounded we have by Theorem \ref{PWK}-(ii) again that $(\ast
)_{1}\leq D_{1}e^{N\left(  \ln\left(  1+\delta\right)  -\delta\right)  /8}$
for some $D_{1}>0$ and $N>N_{\delta}^{x}$ (cf. the proof of Theorem
(\ref{PWK})-(ii)). Now let us consider $(\ast)_{2}.$ From Theorem \ref{PWK} it
follows similar to the proof of Proposition \ref{propc1} that%
\begin{equation}
\int_{\left\vert t-1\right\vert <\delta}(t-1)K_{N}(t,x)dt=\frac{1}{N}-{{\i}
}\frac{c(N/x)}{N/x}+O\left(  \frac{1}{N}+\frac{c(N/x)}{N/x}\right)
^{2}+f_{N,\delta}\label{int1}%
\end{equation}
with $\left\vert f_{N,\delta}\right\vert \leq D_{2}e^{N\left(  \ln\left(
1+\delta\right)  -\delta\right)  /8},$ and by Lemma \ref{ass1} we have that%
\begin{align}
\left\vert \int_{\left\vert t-1\right\vert <\delta}(t-1)^{2}R_{t}^{x}%
K_{N}(t,x)dt\right\vert  &  \leq R^{x}\int_{0}^{\infty}\left\vert
t-1\right\vert ^{2}\left\vert K_{N}(t,x)\right\vert dt\nonumber\\
&  \leq\frac{D_{3}}{N}\exp\left[  O\left(  c^{2}(\frac{N}{x})\frac{x^{2}}%
{N}\right)  \right]  \label{int}%
\end{align}
We thus get by (\ref{int0}), (\ref{int1}), (\ref{int}), and assumption
(\ref{gamma}),%
\[
p_{N}(x)=p(x)+xp^{\prime}(x)\left(  \frac{1}{N}-{\i}\frac{c(N/x)}%
{N/x}\right)  +O_{x}\left(  \frac{1}{N}+\frac{c(N/x)}{N/x}\right)
^{2}+O(N^{-1})
\]
from which the statements follow by taking the real part and taking
(\ref{reee}) into account.
\section{Proof of Proposition~\ref{thm:var}}
For a generic constant $C>1,$ depending on $x$ and changing in this proof from
line to line, we may write
\begin{align*}
\text{Var}\left[  p_{n,N}(x)\right]   &  \lesssim\frac{C^{N}}{n}\frac
{N^{2N+2}}{N^{2N}}\mathbb{E}\left[  \left|  \sum_{k=1}^{N}({\i}X_{1}%
)^{k}\left(  -1\right)  ^{N-k}\left(  2g(N/x)-\mu^{2}\right)  ^{\frac{1}%
{2}k-N}F_{N,k}\right|  ^{2}\right] \\
&  \lesssim\frac{C^{N}}{n}\mathbb{E}\left[  \left|  \sum_{k=1}^{N}({{\i}
}X_{1})^{k}\left(  -1\right)  ^{N-k}\left(  2g(N/x)-\mu^{2}\right)  ^{\frac
{1}{2}k-N}F_{N,k}\right|  ^{2}\right] \\
&  \lesssim\frac{C^{N}}{n}\sum_{k=1}^{N}\mathbb{E}\left[  \left\vert
X_{1}\right\vert ^{2k}\left\vert 2g(N/x)-\mu^{2}\right\vert ^{k-2N}F_{N,k}%
^{2}\right] \\
&  \lesssim\frac{C^{N}}{n}\sum_{k=1}^{N}N^{k-2N}F_{N,k}^{2}\mathbb{E}\left[
\left\vert X_{1}\right\vert ^{2k}\right]  .
\end{align*}
It is not difficult to see that for $l=1,...,N-k+1,$%
\[
\frac{\left(  2\left(  l-1\right)  \right)  !}{2^{l-1}\left(  l-1\right)
!}\leq C^{l}l!
\]
and so from the definition of the Bell polynomials it follows that%
\begin{align*}
F_{N,k}  &  \leq C^{k}B_{N,k}\left(  1!,...,(N-k+1)!\right) \\
&  =C^{k}\left(
\begin{tabular}
[c]{c}%
$N$\\
$k$%
\end{tabular}
\ \ \right)  \left(
\begin{tabular}
[c]{c}%
$N-1$\\
$k-1$%
\end{tabular}
\ \ \right)  \left(  N-k\right)  !\\
&  \leq C^{N}N^{N-k}.
\end{align*}

\section{Proof of Proposition \ref{rates}}\label{sec:rates}
(i): Without loss of generality we may assume that $A>1.$ Since for $k\leq N,$
$N^{-k}\leq k^{-k},$ we get from (\ref{var}),
\begin{equation}
\text{Var}\left[  p_{n,N}(x)\right]  \leq\frac{C^{N}}{n}\sum_{k=1}^{N}%
A^{k}\leq\frac{A}{A-1}\frac{(AC)^{N}}{n}. \label{var1}%
\end{equation}
By substituting $N$ according to (\ref{NC}) into (\ref{var1}) we obtain%
\[
\text{Var}\left[  p_{n,N}(x)\right]  \leq\frac{A}{A-1}\frac{(AC)^{N}}{n}%
=\frac{A}{A-1}\ln^{-2\rho}n,
\]
while for the squared bias we have%
\[
\mathcal{R}_{N}^{2}(x)=O_{x}(N^{-2\rho})=O_{x}\left(  \frac{1}{\ln^{2\rho}%
n}\right)  ,
\]
hence (i) follows.

(ii): In this case (\ref{var}) yields%
\begin{equation}
\text{Var}\left[  p_{n,N}(x)\right]  \leq\frac{C^{N}}{n}\sum_{k=1}^{N}%
A^{k}k^{\left(  b-1\right)  k}\leq\frac{C_{1}^{N}}{n}N^{\left(  b-1\right)  N}
\label{var2}%
\end{equation}
for another constant $C_{1}>1.$ From (\ref{NTH}) it follows straightforwardly
that%
\begin{align*}
\ln N^{\left(  b-1\right)  N+2\rho}  &  =\frac{2\rho+\ln n}{\ln\left(  \left(
2\rho+\ln n\right)  /\left(  b-1\right)  \right)  }\ln\left(  \frac{\left(
2\rho+\ln n\right)  }{\left(  b-1\right)  \ln\left(  \left(  2\rho+\ln
n\right)  /\left(  b-1\right)  \right)  }-\frac{2\rho}{b-1}\right) \\
&  =\frac{2\rho+\ln n}{\ln\left(  \left(  2\rho+\ln n\right)  /\left(
b-1\right)  \right)  }\\
&  \times\ln\left[  \frac{\left(  2\rho+\ln n\right)  /\left(  b-1\right)
}{\ln\left(  \left(  2\rho+\ln n\right)  /\left(  b-1\right)  \right)
}\left(  1-2\rho\frac{\ln\left(  \left(  2\rho+\ln n\right)  /\left(
b-1\right)  \right)  }{2\rho+\ln n}\right)  \right]
\end{align*}%
\[
=\frac{2\rho+\ln n}{\ln\left(  \left(  2\rho+\ln n\right)  /\left(
b-1\right)  \right)  }\ln\frac{\left(  2\rho+\ln n\right)  /\left(
b-1\right)  }{\ln\left(  \left(  2\rho+\ln n\right)  /\left(  b-1\right)
\right)  }-2\rho+O\left(  \frac{\ln\ln n}{\ln n}\right)
\]%
\[
=\ln n-\left(  2\rho+\ln n\right)  \frac{\ln\ln\left(  \left(  2\rho+\ln
n\right)  /\left(  b-1\right)  \right)  }{\ln\left(  \left(  2\rho+\ln
n\right)  /\left(  b-1\right)  \right)  }+O\left(  \frac{\ln\ln n}{\ln
n}\right)  .
\]

On the other hand, from (\ref{NTH}),%
\[
\ln C_{1}^{N}\leq\frac{\left(  2\rho+\ln n\right)  \ln C_{1}}{\left(
b-1\right)  \ln\left(  \left(  2\rho+\ln n\right)  /\left(  b-1\right)
\right)  }%
\]
So, for $n\rightarrow\infty,$%
\begin{gather*}
\ln\left(  C_{1}^{N}N^{\left(  b-1\right)  N+2\rho}\right)  \leq\ln n+O\left(
\frac{\ln\ln n}{\ln n}\right) \\
-\,\underset{\rightarrow+\infty}{\underbrace{\left(  2\rho+\ln n\right)
\frac{\ln\ln\left(  \left(  2\rho+\ln n\right)  /\left(  b-1\right)  \right)
}{\ln\left(  \left(  2\rho+\ln n\right)  /\left(  b-1\right)  \right)  }}}\\
\times\left(  1-\underset{\rightarrow0}{\underbrace{\frac{\ln C_{1}}{\left(
b-1\right)  \ln\ln\left(  \left(  2\rho+\ln n\right)  /\left(  b-1\right)
\right)  }}}\right)  ,
\end{gather*}
hence%
\[
C_{1}^{N}N^{\left(  b-1\right)  N+2\rho}\leq2n
\]
We thus obtain by (\ref{var2}),
\[
\operatorname{Var}(p_{n,N}(x))\leq2N^{-2\rho}\lesssim\frac{\ln^{2\rho}\ln
n}{\ln^{2\rho}n}%
\]
while%
\[
\mathcal{R}_{N}^{2}(x)=O_{x}(N^{-2\rho})=O_{x}\left(  \frac{\ln^{2\rho}\ln
n}{\ln^{2\rho}n}\right)
\]
which gives (ii).

\bibliographystyle{plain}
\bibliography{est_subbm_bibliography}

\end{document}